\documentclass[10pt]{article}
 
 \usepackage{amsmath,amsthm,amssymb}
 \usepackage{makeidx,epsfig,lscape}
 \usepackage{color,colortbl}
 \usepackage{fancyhdr}
 \usepackage{xcolor,pict2e}
 \usepackage{booktabs,array}
 \usepackage{mathrsfs}

 \newcolumntype{Q}{ >{$\displaystyle}l <{$}}
 \newcolumntype{A}{ >{$}c <{$}}
\newcolumntype{B}{ >{$} r <{$} @{} >{${}} l <{$} } 
 \definecolor{myaqua}{rgb}{0.0,0.5,0.55}
 \definecolor{lightaqua}{rgb}{0.75,0.95,0.95}

\usepackage[hyphens]{url}
\usepackage[colorlinks = true,
            linkcolor = myaqua,
            urlcolor  = blue,
            citecolor = myaqua]{hyperref}

\newtheorem{theorem}{Theorem}[section]
\newtheorem{lemma}[theorem]{Lemma}

\newtheorem{corollary}[theorem]{Corollary}

\theoremstyle{definition}

\newtheorem{example}[theorem]{Example}

\theoremstyle{remark}
\newtheorem{remark}[theorem]{Remark}
\numberwithin{equation}{section}

\newcommand{\re}{\mathrm{Re}}
\newcommand{\Li}{\mathrm{Li}}
\newcommand{\intervallegenerique}[4]{\mathopen{}\mathclose\bgroup\left#1
                                     #2\mathclose{}\mathpunct{},#3
                                     \aftergroup\egroup\right#4}
\newcommand{\intervalle}[2]{\intervallegenerique{[}{#1}{#2}{]}}
\newcommand{\intervallefo}[2]{\intervallegenerique{[}{#1}{#2}{)}}
\newcommand{\intervalleof}[2]{\intervallegenerique{(}{#1}{#2}{]}}
\newcommand{\intervalleoo}[2]{\intervallegenerique{(}{#1}{#2}{)}}

\newcommand{\abs}[1]{\left\lvert#1\right\rvert}

\DeclareRobustCommand{\eulerian}{\genfrac<>{0pt}{}}

\title{The Riemann Hypothesis: A Qualitative \\[2mm] Characterization of the Nontrivial Zeros of the \\[2mm] Riemann Zeta Function Using Polylogarithms}

\author{Lazhar Fekih-Ahmed\\
\'{E}cole Nationale d'Ing\'{e}nieurs de Tunis,\\
 BP 37, Le Belv\'{e}d\`{e}re, 1002,\\
Tunis, Tunisia\\
Email: \href{mailto:lazhar.fekihahmed@enit.rnu.tn}{\color{blue}{\underline{\smash{lazhar.fekihahmed@enit.rnu.tn}}}}\\}
\date{Date: August 22, 2016}

\begin{document}

\maketitle
 
\begin{abstract}
\noindent We formulate a parametrized uniformly absolutely globally convergent series of 
$\zeta(s)$ denoted by $Z(s,x)$. When expressed in closed form, it is given by
\begin{align}\nonumber
Z(s,x)&=(s-1)\zeta(s)+\int_{x}^{1}\Li_{s}\left( \frac{z}{z-1}\right)\,dz,
\end{align}
where $\Li_{s}(x)$ is the polylogarithm function. As 	an immediate  first application of the new parametrized series, a new expression of $\zeta(s)$ follows: $\displaystyle{(s-1)\zeta(s)=-\int_{0}^{1}\Li_{s}\left( \frac{z}{z-1}\right)\,dz}$. As a second important application, using the functional equation and exploiting uniform convergence of the series defining $Z(s,x)$, we have for any  non-trivial zero $s$ 
\begin{align}\nonumber
&\frac{(s-1)\zeta(s)}{-s\zeta(1-s)}\text{~nonzero  \& finite}\quad\Rightarrow\quad
\lim\limits_{x \to 1}\frac{\abs{\int_{x}^{1}\Li_{s}\left( \frac{z}{z-1}\right)\,dz}}
{\abs{\int_{x}^{1}\Li_{1-s}\left( \frac{z}{z-1}\right)\,dz}}\text{~nonzero \&
 finite}.
\end{align}
The necessary condition from the last statement translates into qualitative information on the non-trivial zeros of $\zeta(s)$, which coincides with the Riemann Hypothesis.    
\end{abstract}


\section{Introduction}

\noindent The Riemann zeta function  defined  by the
Dirichlet series

\begin{align}\label{sec1-eq1} 
\zeta(s)=&\sum_{n=1}^{\infty}\frac{1}{n^{s}}
\end{align}
 is a convergent well-defined series for $\re(s)>1$, and can be analytically
continued to the whole complex plane with one singularity, a
simple pole with residue $1$ at $s=1$. To numerically evaluate $\zeta(s)$ at a value of
$s$ outside the half plane $\re(s)>1$, the original defining
series of (\ref{sec1-eq1}) is not appropriate since it not convergent.

Although, there exists many series which analytically continue the original series (\ref{sec1-eq1}) to the whole complex plane, and therefore permit
the calculation of $\zeta(s)$ for any value of $s\in \mathbb{C}$, the series found in \cite{lazhar:zeros} and originally due to Ser \cite{ser}, has the special feature that it is extremely slowly convergent. While computationally this is not a desired feature, we will see in this paper that the series leads to very interesting theoretical results. In Section~\ref{sec2}, we recall the globally convergent series studied in \cite{lazhar:zeros}. In Section~\ref{sec3}, we parametrize this series using Abel type parametrization and  extract some its main properties. The main section of the paper is Section~\ref{sec4} in which  we apply the parametrized series to qualitatively characterizing the nontrivial zeros of $\zeta(s)$. In Section~\ref{sec5}, we investigate the possibility of obtaining similar results to other well-known globally convergent series of $\zeta(s)$, and in Section~\ref{sec6}, we conclude with several remarks.

\section{A Globally Convergent Series for \texorpdfstring{$(s-1)\zeta(s)$}{(s-1)zeta(s)}}\label{sec2}

 \noindent

This section is based on the author's  papers \cite{lazhar:zeros,lazhar:hurwitz}. Define

\begin{equation}
S_n(s)=\sum_{k=0}^{n-1}(-1)^{k}\binom{n-1}{k}(k+1)^{-s}.\label{sec2-eq0}
\end{equation}

The globally convergent representation of $(s-1)\zeta(s)$ that we use in this paper is summarized in the following

\begin{theorem}\label{sec2-thm1}
Let $S_n(s)$ be defined by (\ref{sec2-eq0}) and let $\Delta_n(s)=\sum_{k=1}^{n}(-1)^{k}\binom{n}{k}k^{1-s}$, then for all $s\in \mathbb{C}$, we have

\begin{align}\label{sec2-eq1}
(s-1)\zeta(s)&=\sum_{n=1}^{\infty}\frac{1}{n+1}S_{n}(s)=-\sum_{n=1}^{\infty}\frac{1}{n(n+1)}\Delta_{n}(s).
\end{align}

Moreover, the series (\ref{sec2-eq1}) is uniformly and absolutely convergent on compact sets of the $s$ plane.

\end{theorem}

\begin{proof}

We can rewrite $S_{n}(s)$  as

\begin{eqnarray}
S_{n}(s) &=&
\frac{1}{\Gamma(s)}\int_{0}^{\infty}\sum_{k=0}^{n-1}(-1)^{k}\binom{n-1}{k}
   e^{-kt}t^{s-1}\,dt\nonumber \\
  &=&
  \frac{1}{\Gamma(s)}\int_{0}^{\infty}(1-e^{-t})^{n-1}e^{-t}t^{s-1}\,dt,\label{sec2-eq2}
\end{eqnarray}

since we know that

\begin{equation}\label{sec2-eq3}
n^{-s}=\frac{1}{\Gamma(s)}\int_{0}^{\infty}e^{-nt}t^{s-1}\,dt
\end{equation}

is a valid formula for $\displaystyle{\re(s)>0}$, and since

\begin{equation}\label{sec2-eq4}
\sum_{k=0}^{n-1}(-1)^{k}\binom{n-1}{k}
e^{-kt}=e^{-t}(1-e^{-t})^{n-1}.
\end{equation}

The proof consists in evaluating the sum

\begin{equation}\label{sec2-eq5}
\sum_{n=1}^{\infty}\frac{S_{n}(s)}{n+1}.
\end{equation}

Without worrying about interchanging sums and integrals for the
moment, we have

\begin{eqnarray}\label{sec2-eq6}
\sum_{n=1}^{\infty}\frac{S_{n}(s)}{n+1}&=&\frac{1}{\Gamma(s)}\sum_{n=1}^{\infty}\int_{0}^{\infty}\frac{(1-e^{-t})^{n-1}}{n+1}e^{-t}t^{s-1}\,dt\\
&=&\frac{1}{\Gamma(s)}\int_{0}^{\infty}\sum_{n=1}^{\infty}\frac{(1-e^{-t})^{n-1}}{n+1}e^{-t}t^{s-1}\,dt.\label{sec2-eq7}
\end{eqnarray}

But for $0<t<\infty$, the logarithmic series

\begin{equation}\label{sec2-eq8}
t=-\log(1-(1-e^{-t}))=\sum_{n=1}^{\infty}\frac{(1-e^{-t})^{n}}{n}
\end{equation}
is valid. After dividing both sides by $(1-e^{-t})^2$ and rearranging terms, we obtain the following identity

\begin{align}\label{sec2-eq9}
\frac{te^{-t}}{(1-e^{-t})^2}-\frac{e^{-t}}{1-e^{-t}}=\sum_{n=1}^{\infty}\frac{(1-e^{-t})^{n-1}e^{-t}}{n+1}.
\end{align}

Therefore,

\begin{align}\label{sec2-eq10}
\sum_{n=1}^{\infty}\frac{S_{n}(s)}{n+1}&= \frac{1}{\Gamma(s)}\int_{0}^{\infty}\bigg
(\frac{t}{(1-e^{-t})^2}-\frac{1}{1-e^{-t}}\bigg )
e^{-t}t^{s-1}\,dt.
\end{align}

Now, by observing that
\begin{equation}\label{sec2-eq11}
\frac{d}{dt}\bigg (\frac{-te^{-t}}{1-e^{-t}}\bigg )=\frac{te^{-t}}{(1-e^{-t})^2}-\frac{e^{-t}}{1-e^{-t}},
\end{equation}

we can perform an  integration by parts in (\ref{sec2-eq10}) when
$\re(s)>1$. The integral in the right hand side of
(\ref{sec2-eq10}) becomes

\begin{equation}\label{sec2-eq12}
\frac{s-1}{\Gamma(s)}\int_{0}^{\infty} \frac{t^{s-1}}{e^{t}-1}
\,dt=(s-1)\zeta(s).
\end{equation}

Thus,

\begin{equation}\label{sec2-eq13}
\sum_{n=1}^{\infty}\frac{S_{n}(s)}{n+1}=(s-1)\zeta(s),
\end{equation}

and this prove the theorem when $\re(s)>1$. However, formula
(\ref{sec2-eq13}) remains valid for $\re(s)>0$ since the integral
(\ref{sec2-eq10}) is well-defined for $\re(s)>0$. This prove the validity of the series representation for $\re(s)>0$. 

In order to prove that the series is valid for all $s\in \mathbb{C}$, we can either use the method of analytic continuation using repeated integration by parts of the right hand side of 
(\ref{sec2-eq10}) as it was carried out in 
\cite[Corollary~2.2]{lazhar:hurwitz}, or we can explicitly bound the sum $S_{n}(s)$ when $n$ is large. We opt for the second method by establishing a lemma which is interesting in its own right.  The lemma provides an estimate of the exact asymptotic order of growth of
$S_{n}(s)$ when $n$ is large.

\begin{lemma}\label{sec2-lem1}
For $n$ large enough, $\displaystyle{S_{n}(s)\thicksim \frac{1}{n(\log
n)^{1-s}\Gamma(s)}}$ when $s\in\mathbb{C}$, 
$s\notin \{0,\pm 1,\pm 2,\cdots\}$, and for
$ k \in \{0,1,2,\cdots\}$ and $n$ large enough, $\displaystyle{S_{n}(k)\thicksim
\frac{(\log{n})^{k-1}}{n(k-1)!}}$ ($k\ne 0$) and $S_{n}(-k)=0$.
\end{lemma}
\begin{proof}

By looking at the definition of
$S_{n}(s)$, we can see that $S_{n}(s)$ are the Stirling
numbers of the second kind modulo a multiplicative factor when
$s\in \{0,-1,-2,\cdots\}$. Therefore, when $s=-k$, $k$ a positive
integer, $S_{n}(-k)=0$ for $n\ge k+1$; hence, they are eventually zero for $n$ large enough.

For $s\notin \{0,-1,-2,\cdots\}$, an asymptotic estimate of
$S_{n}(s)$ can be obtained for $n$ large. Indeed, by putting 
$k=m-1$ in (\ref{sec2-eq0}), we have by definition

\begin{eqnarray}
S_{n}(s)&=&\sum_{m=1}^{n}\binom{n-1}{m-1}(-1)^{m-1}m^{-s}=
\sum_{m=1}^{n}\frac{m}{n}\binom{n}{m}(-1)^{m-1}m^{-s} \nonumber\\
&=& \frac{-1}{n}\sum_{m=1}^{n}\binom{n}{m}(-1)^{m}m^{1-s}=
\frac{-1}{n}\Delta_{n}(s-1),\label{sec2-eq14}
\end{eqnarray}

where $\displaystyle{\Delta_{n}(\lambda)\triangleq
\sum_{m=1}^{n}\binom{n}{m}(-1)^{m}m^{-\lambda}}$.

The asymptotic expansion of sums of the form
$\displaystyle{\Delta_{n}(\lambda)}$, with $\lambda\in\mathbb{C}$
 has been given in ~\cite{flajolet:rice}. With a slight modification of notation, the
authors in ~\cite{flajolet:rice} have shown that
$\Delta_{n}(\lambda)$ has the following asymptotics when $n$ is large

\begin{equation}\label{sec2-eq15}
\Delta_{n}(\lambda)\thicksim\frac{-(\log{n})^{\lambda}}{\Gamma(1+\lambda)}
\end{equation}

when $\lambda$ is nonintegral ~\cite[Theorem 3]{flajolet:rice}, and the following expansion

\begin{equation}\label{sec2-eq16}
\Delta_{n}(k-1)\thicksim
-\frac{(\log{n})^{k-1}}{(k-1)!}
\end{equation}

when $k\in \{1,2,\cdots\}$ ~\cite[Eq. (9)]{flajolet:rice}. Thus, when $s$ is nonintegral, we can apply these results to $\Delta_{n}(\lambda)$ with $\lambda=s-1$
to get
\begin{equation}\label{sec2-eq17}
\Delta_{n}(s-1) \thicksim \frac{-(\log{n})^{s-1}}{\Gamma(s)},
\end{equation}

which leads to the result
\begin{equation}\label{sec2-eq18}
S_{n}(s) \thicksim\frac{1}{n(\log n)^{1-s}\Gamma(s)}.
\end{equation}

Similarly, when $s=k\in \{1,2,\cdots\}$ we obtain the following expansion

\begin{equation}\label{sec2-eq19}
S_{n}(k)\thicksim
\frac{(\log{n})^{k-2}}{(k-1)!}.
\end{equation}
The asymptotic estimates  (\ref{sec2-eq18})
and (\ref{sec2-eq19}) are valid for $n$ large enough and for all
$s$ such that $\Re(s)>0$. This proves the lemma.
\end{proof}

To prove that our series (\ref{sec2-eq1}) is well-defined and valid for all $s\in\mathbb{C}$, we note that the logarithmic test of series implies that our series
is dominated by an absolutely uniformly convergent series for all finite $s$
such that $\re(s)>0$. Now, by Weierstrass theorem of the
uniqueness of analytic continuation, the function
$(s-1)\zeta(s)$, initially well-defined by (\ref{sec1-eq1}) for $\re(s)>1$, can be extended outside of the domain $\re(s)>1$
and that it does not have any singularity when $\re(s)>0$.
Moreover, by repeating the same process for $\re(s)>-k$, $k\in
\mathbb{N}$, it is clear that the series defines an analytic
continuation of $\zeta(s)$ valid for all $s\in\mathbb{C}$.

To finish the proof of the theorem, we now need to justify the interchange of
summation and integration in (\ref{sec2-eq7}). The interchange  is indeed valid because the series

\begin{equation}\label{sec2-eq20}
\sum_{n=1}^{\infty}\int_{0}^{\infty}\frac{(1-e^{-t})^{n-1}}{n+1}e^{-t}t^{s-1}\,dt
\end{equation}

converges absolutely and uniformly for $0<t<\infty$. To prove
this, it suffices to  show uniform convergence for the dominating
series

\begin{equation}\label{sec2-eq21}
\sum_{n=1}^{\infty}\int_{0}^{\infty}\frac{(1-e^{-t})^{n-1}}{n+1}e^{-t}t^{\sigma-1}\,dt,
\end{equation}
where $\sigma=\re(s)$. Indeed, let $K={\rm max}((1-e^{-t})^{n-1}e^{-t/2})$, $0<t<\infty$.
A straightforward calculation of the derivative shows that

\begin{equation}\label{sec2-eq22}
K=(1-\frac{1}{2n-1})^{n-1}\frac{1}{\sqrt{2n-1}},
\end{equation}

and  that the maximum  is attained when $e^{-t}=\frac{1}{2n-1}$. Now, for $n\ge 2$, we have

\begin{eqnarray}
\int_{0}^{\infty}(1-e^{-t})^{n-1}e^{-t}t^{\sigma-1}\,dt&=&
\int_{0}^{\infty}(1-e^{-t})^{n-1}e^{-t/2}
e^{-t/2}t^{\sigma-1})\,dt\nonumber\\
&\le &K \int_{0}^{\infty}e^{-t/2}t^{\sigma-1}\,dt\nonumber\\
&=&(1-\frac{1}{2n-1})^{n-1}\frac{2^{\sigma}\Gamma(\sigma)}{\sqrt{2n-1}}\nonumber\\
&\le&\frac{K^{\prime}}{\sqrt{2n-1}}.\label{sec2-eq23}
\end{eqnarray}

The last inequality implies that each term of the dominating
series is bounded by $K^{\prime}/(n+1) \sqrt{2n-1}$. Thus the
dominating series converges by the comparison test. This completes
the proof of the theorem.
\end{proof}

\section{Power Series to Parametrize \texorpdfstring{$\zeta(s)$}{zeta(s)}}\label{sec3}

 \noindent

In order to shed some light on the nature of convergence of a series, the original series are often parametrized using a new complex variable $x$. The most useful and most natural parametrization is the so-called Abel parametrization. The new parametrized series is a power series, and the utility of Abel's construction via Abel's theorem for power series is to relate the limit (as $x$ approaches 1) of the power series the to the sum of the original series. Note that the original series may be convergent or divergent. For divergent series it is known as Abel's summation method. Obviously, if the original series is uniformly absolutely-convergent then it is Abel summable.

In the following subsections, we study two power series associated with the zeta function: the polylogarithm and a power series associated with the numerical series of the previous section. The advantages of using the second power series will be apparent in the rest of the paper since our interest is mainly directed to the non-trivial zeros which are known to reside in the critical strip.

\subsection{Parametrization of the Defining Series: The Polylogarithm Function}\label{subsec3_1}

 \noindent
The most well-know and the earliest parametrization attached to the original definition of the Riemann zeta function (\ref{sec1-eq1}) is the polylogarithm function, denoted by $\Li_{s}(x)$. It is defined by the power series

\begin{equation}\label{sec3-eq1}
\Li_{s}(x)=\sum_{n=1}^{\infty}\frac{x^n}{n^{s}}.
\end{equation}

The definition is valid for all complex values $s$  and all
complex values of $x$ such that $|x|<1$ but the series is convergent
for $x=1$ only when $\re(s)>1$.

Using the identity

\begin{equation}\label{sec3-eq2}
\frac{1}{n^s}=\frac{1}{\Gamma(s)}\int_{0}^{\infty}e^{-nt}t^{s-1}\,dt,
\end{equation}

equation (\ref{sec3-eq1}) can be rewritten as

\begin{equation}\label{sec3-eq3}
\Li_{s}(x)=\frac{x}{\Gamma(s)}\int_{0}^{\infty}\frac{t^{s-1}}{e^{t}-x}\,dt.
\end{equation}

The integral in  (\ref{sec3-eq3}) is called Appell's integral or
Jonqui\`{e}re's integral. It defines $\Li_{s}(x)$ not only in the
unit circle but also in the whole slit plane $\mathbb{C}\setminus
\intervallefo{1}{\infty}$ provided that $\re(s)>0$.

By a change of variable, $u=e^{-t}$, the polylogarithm function can be expressed as

\begin{align}
\Li_{s}(x)&=\frac{1}{\Gamma(s)}\int_{0}^{1}\frac{(-\ln{u})^{s-1}}{\frac{1}{x}-u}\,du, x\ne 0; \Li_{s}(0)=0.\label{sec3-eq4}
\end{align}

In formula (\ref{sec3-eq4}), the integral without the gamma factor is a a Cauchy-type singular integral. By putting $w=\frac{1}{x}$, we can deduce the local and global properties of $\Li_{s}(x)=\Li_{s}(\frac{1}{w})$. Indeed, the integral 

\begin{align}
\Omega(w)&=\int_{0}^{1}\frac{(-\ln{u})^{s-1}}{w-u}\,du, \label{sec3-eq5}
\end{align}

is a  singular integral of Cauchy type with an absolutely integrable density function possessing a logarithmic singularity \cite[p. 62]{gakhov}. Both the integrand and the integral are multi-valued. The
function $\Omega(w)$ vanishes at infinity, and is a holomorphic function in all 
the complex plane except possibly on the cut $\intervalle{0}{1}$ which constitutes
a singular segment (i.e. a segment of discontinuity). But since the density  $(-\ln{u})^{s-1}$ is holomorphic in a neighborhood of $\intervalle{0}{1}$, the
singular segment is not essential in the sense that the line integral can be deformed into an equal curvilinear  integral from $0$ to $1$, and the only truly singular points are $0$ and
$1$. The function $\Omega(w)$ is, however, not single-valued and possesses periods.

Going back to the polylogarithm, we conclude that $\Li_{s}(x)$ is a holomorphic function in the $x$ plane cut along $\intervallefo{1}{\infty}$ and that the only truly singular points are $1$ and $\infty$. Moreover, the polylogarithm is well-defined at $0$. Let's denote the principal branch of $\Li_{s}(x)$ by $\Li_{s}(x)^{\star}$; that is
the branch which is equal to the original defining series of $\Li_{s}(x)$ around $x=0$ (sometimes called the holomorphic germ at $0$). If we
further denote by $\ln(x)$ the principal branch of the logarithm (i.e. with a
branch cut along the negative real axis), the general expression of $\Li_{s}(x)$ is given by

\begin{align}
\Li_{s}(x)&=\Li_{s}(x)^{\star}\pm \frac{2k\pi i}{\Gamma(s)}\left(\ln(x)\pm 2m\pi i\right)^{s-1} \label{sec3-eq6}
\end{align}

if the integral (\ref{sec3-eq5}) is taken a along a path which winds $k$ times around $x=1$ and $m$  times around $x=0$. Thus, $x=0$ is a singular point for all the branches of $\Li_{s}(x)$ other than the principal branch. 

The last expression also shows that the polylogarithm extends analytically to a global single-valued analytic function if defined properly on a Riemann surface constructed using adequate pasting of copies of the slit plane $\mathbb{C}^{*}=\mathbb{C}\setminus
\left\{\intervalleof{-\infty}{0}\cup \intervallefo{1}{\infty}\right\}$ and a copy of the main branch not split at $x=0$ but with a cut along $\intervallefo{1}{\infty}$.

An alternative way of extending analytically the original defining series 
of $\Li_{s}(x)$ is by use of the inversion formulas, also known as the Lerch-Jonqui{\`e}re formulas:

\begin{theorem}[Lerch-Jonqui{\`e}re formulas, \cite{jonquiere,wikipedia}]\label{sec3-thm1}
For all complex $s$, $\Li_{s}(x)$ satisfies the two inversion formulas:

\begin{enumerate}
\item  For complex $x\notin \intervalle{0}{1}$:

\begin{align}
\Li_{s}(x)+(-1)^{s}\Li_{s}\left(\frac{1}{x}\right)&=\frac{(2\pi i)^{s}}{\Gamma(s)}\zeta\left(1-s,\frac{1}{2}+\frac{\ln(-x)}{2\pi i}\right).\label{sec3-eq7}
\end{align}

\item  For complex $x\notin \intervallefo{1}{\infty}$:
\begin{align}
\Li_{s}(x)+(-1)^{s}\Li_{s}\left(\frac{1}{x}\right)&=\frac{(2\pi i)^{s}}{\Gamma(s)}\zeta\left(1-s,\frac{1}{2}-\frac{\ln(-\frac{1}{x})}{2\pi i}\right).\label{sec3-eq8}
\end{align}

\item For $x\notin \intervallefo{0}{\infty}$ and by taking $\ln(-x)+\pi i=\ln(x)$, the two expressions (\ref{sec3-eq7})-(\ref{sec3-eq8}) agree because $\ln(-x)=-\ln\left(-\frac{1}{x}\right)$, and we have
\begin{align}
\Li_{s}(x)+(-1)^{s}\Li_{s}\left(\frac{1}{x}\right)&=\frac{(2\pi i)^{s}}{\Gamma(s)}\zeta\left(1-s,\frac{\ln(x)}{2\pi i}\right).\label{sec3-eq9}
\end{align}
\end{enumerate}
\end{theorem}

With the above inversion relations, $\Li_{s}(x)$ can be considered known in the whole complex plane.

We now analyze the local behavior of the principal branch of $\Li_{s}(x)$ around the singular points $1$ and $\infty$ (or around $0$ for the branches other than the principal branch). The direct method is to use  the Cauchy type integral representations 
(\ref{sec3-eq4}) or (\ref{sec3-eq5}) of the polylogarithm in order to determine the local behavior of the integrals  near the endpoints of the contour of integration $w=0$ and $w=1$ (or $x=\infty$ and $x=1$).

For the remaining sections of the paper, we need to know the behavior of 
$\Li_{s}(x)$ near $\infty$. So, we only describe the behavior around the lower 
point of the contour, i.e $w=0$, the other endpoint is similar (for details see  \cite{gakhov}).

For the special endpoint $w=0$ (or $x=\infty$) and concerning $\Li_{s}(x)$, this 
has been carried out in \cite{costin} using integration by parts and in 
\cite{gawronski} using Watson's lemma for loop integrals. The methods of 
\cite{gakhov} (see also the references therein) are of a different nature. They 
are more general and more systematic, and deal with a wide variety of
 Cauchy type integrals using the Sokhotski-Plemelj formulas. When $s$ is an 
integer, the Cauchy type integral (\ref{sec3-eq5}) can be expressed very
 elegantly and in finite form. But when $s$ is not an integer, the local 
expression of the integral cannot be given in finite form and thus does not 
offer any advantage over the asymptotic behavior of $\Li_{s}(x)$ given in
\cite{costin} or \cite{gawronski}. \footnote{The asymptotic formula given in 
\cite{costin} is $\Li_{s}(z)=-\frac{(\ln(z))^{s}}{\Gamma(s+1)}+o(1), |z|\to \infty, z \notin \intervallefo{1}{\infty}$ whereas the formula in \cite{gawronski} is
 $\Li_{s}(z)=-\frac{(\ln(-z))^{s}}{\Gamma(s+1)}+o(1), |z|\to \infty, z\notin \intervallefo{1}{\infty}$. Both formulas are essentially the same if we take 
$\ln(-z)+\pi i=\ln(z)$, $z\notin \intervallefo{0}{\infty}$. Using the methods 
of \cite{gakhov}, one can show that $\Omega(w)=-\frac{(-\ln(w))^{s}}{\Gamma(s+1)}+\Omega_0(w)+\Omega_1(w)$, where $\Omega_1(w)$ is analytic at $w=0$ and  $\Omega_0(w)=O((-\ln(w))^{s-1})$. }

The indirect way, and probably the easiest, to obtain an asymptotic the behavior of $\Li_{s}(x)$ when $x$ is large is to use the global ready-to-use Lerch-Jonqui{\`e}re formulas of Theorem~\ref{sec3-thm1}. They  give the behavior for both integer and  fractional values of $s$ at the expense of evaluating the asymptotic behavior of the hurwitz zeta function. We obtain the following

\begin{corollary}\label{sec3-cor1}
For complex numbers $s$ with $\re(s)>0$,

\begin{align}
\Li_{s}(x)&=-\frac{(\ln(-x))^{s}}{\Gamma(s+1)}+o(1), |x|\to \infty, x\notin \intervallefo{0}{\infty}.\label{sec3-eq10}
\end{align}
\end{corollary}
\begin{proof}
Using Hermite's representation of the Hurwitz zeta function \cite[p. 106]{lindelof:1905}
\begin{align}
\zeta(1-s,a)&=\frac{a^{s-1}}{2}-\frac{a^{s}}{s}+
\int_{0}^{\infty}\frac{(a-it)^{s-1}-(a+it)^{s-1}}{i(e^{2\pi t}-1)}\,dt, s\ne 0, \re(a)>0
\label{sec3-eq11}
\end{align}
with $a=\frac{1}{2}+\frac{\ln(-x)}{2\pi i}$, then replacing  the numerator of the integrand by its Taylor expansion and using the formula

\begin{align}
\int_{0}^{\infty}\frac{t^{2j-1}}{e^{2\pi t}-1}\,dt,&=(-1)^{j-1}\frac{B_{2j}}{4j},\label{sec3-eq12}
\end{align}

where $B_{2j}$ are the Bernoulli numbers, we get the well-known asymptotic expansion at infinity

\begin{align}
\zeta(1-s,a)&=\frac{a^{s-1}}{2}-\frac{a^{s}}{s}+\sum_{j=1}^{m-1}\frac{B_{2j}\Gamma(2j-s)}{(2j)!\Gamma(1-s)a^{2j+2-s}}+O(\frac{1}{a^{2m+2-s}})\label{sec3-eq13}
\end{align}

For the details of obtaining the above expansion, see (\cite[pp. 290-291]{olver}). Now, the inversion formula 
(\ref{sec3-eq7}) in conjunction with the estimates

\begin{align}\label{sec3-eq14}
\left(\frac{1}{2}+\frac{\ln(-x)}{2\pi i}\right)^{s}&=\frac{(\ln(-x))^{s}}{(2\pi i)^{s}}(1+o(1)), |x|\to \infty\\
\Li_{s}\left(\frac{1}{x}\right)&=o(1), |x|\to\infty\label{sec3-eq15}
\end{align}
yield the desired result.

\end{proof}

We end this section with two remarks:

\begin{remark}\label{sec3_rem1}

Corollary~\ref{sec3-cor1} can also be proved using the following asymptotic expansion which holds for all $s$ and $|x|$ large enough \cite{wikipedia}:
\begin{align}
\Li_s(x)&=\frac{\pm i\pi}{\Gamma(s)}\left(\ln(-x)\pm i\pi\right)^{s-1}-\sum_{k=0}^{\infty}(-1)^{k}(2\pi)^{2k}\frac{B_{2k}}{(2k)!}\frac{\left(\ln(-x)\pm i\pi\right)^{s-2k}}{\Gamma(s+1-2k)},
\label{sec3-eq16}
\end{align}
where $B_{2k}$ are the Bernoulli numbers. Extracting the $k=0$ term from the infinite sum, we get
\begin{align}
\Li_{s}(x)&\sim -\frac{(\ln(-x))^{s}}{\Gamma(s+1)}, |x|\to \infty.\label{sec3-eq17}
\end{align}
\end{remark}
 
\begin{remark}\label{sec3_rem2}
We can find the relation (\ref{sec3-eq6}) when $k=1$ and $m=0$ (i.e. we cross 
the cut $\intervallefo{1}{\infty}$ once and never making a complete circle 
around $0$) using the inversion formula (\ref{sec3-eq8}). Indeed, when 
$x\in \intervallefo{1}{\infty}$, $\frac{1}{x}$ is not on the cut and therefore 
by continuity $\lim\limits_{\epsilon \to 0} \left\{\Li_{s}\left(\frac{1}{x+ i\epsilon}\right)-\Li_{s}\left(\frac{1}{x- i\epsilon}\right)\right\}=0$. Consequently,

\begin{align}
\lim\limits_{\epsilon \to 0} \left\{\Li_{s}\left(x+ i\epsilon\right)-\Li_{s}\left(x- i\epsilon\right)\right\}&=\frac{(2\pi i)^{s}}{\Gamma(s)}\left\{\zeta\left(1-s,\frac{\ln(x)}{2\pi i}\right)-\zeta\left(1-s,\frac{\ln(x)}{2\pi i}+1\right)\right\}\nonumber \\
&=\frac{2\pi i}{\Gamma(s)}\left(\ln(x)\right)^{s-1}.\label{sec3-eq18}
\end{align}

\end{remark}

\subsection{Parametrization of the Uniformly Absolutely-Convergent Series (\ref{sec2-eq1})}\label{subsec3_2}
 
\noindent The function $\zeta(s)$  which is the original defining series associated with  
the polylogarithm $\Li_{s}(x)$ is convergent only when $\re(s)>1$. In this 
section, we define an Abel type parametrized series associated with 
the uniformly absolutely-convergent series defined in section~\ref{sec2}, and which  is valid when $\re(s)>0$. 
For complex $x$ such that $|x|<1$, we define

\begin{align}
Z(s,x)&=\sum_{n=1}^{\infty}\frac{S_{n}(s)}{n+1}x^{n+1}\label{sec3-eq19}
\end{align}

be the Abel parametrization of the globally convergent series 
(\ref{sec2-eq1}). Note that the exponent of the variable $x$ is not $n$ but 
$n+1$ for reasons that will be apparent shortly. It can be easily seen from Lemma~\ref{sec2-lem1} that the radius of absolute uniform convergence of the series (\ref{sec3-eq19}) is 1.

Define  

\begin{align}
\phi(s,x)&=\sum_{n=1}^{\infty}S_{n}(s)x^{n}
.\label{sec3-eq20}
\end{align}

Similar to $Z(s,x)$, the radius of convergence of $\phi(s,x)$ is 1, and uniform
absolute convergence follows easily by comparison: the series
$\sum_{n=1}^{\infty}S_{n}(s)x^{n}$ clearly converges absolutely for
every point in $S\times \{x:\abs{x}<1\}$, where $S$ is any compact set of the $s$-plane.

Moreover,  within the circle of convergence  $Z(s,x)$ is  an antiderivative of 
$\phi(s,x)$, and we have

\begin{align}
Z(s,x)&=\int_{0}^{x}\sum_{n=1}^{\infty}S_{n}(s)t^{n}\,dt=\int_{0}^{x}\phi(s,t)\,dt.  \label{sec3-eq21}
\end{align}

Applying Abel's limit theorem, which is a consequence of uniform convergence, to the integrated series (\ref{sec3-eq21}), we can see that 

\begin{align}
\lim\limits_{x
\to 1}Z(s,x)=&\sum_{n=1}^{\infty}\frac{S_{n}(s)}{n+1}=(s-1)\zeta(s).\label{sec3-eq22}
\end{align}

Thus, the upper limit in the integral sign can be equal to 1 despite the fact that 
$\phi(s,x)=\sum_{n=1}^{\infty}S_{n}(s)x^{n}$ may not be convergent at $x=1$, a point on the boundary of the region of absolute and uniform convergence. Therefore, we can write

\begin{align}
(s-1)\zeta(s)&=\int_{0}^{1}\phi(s,t)\,dt.  \label{sec3-eq23}
\end{align}

It turns out that $\phi(s,x)$ can be conveniently expressed in terms of  the polylogarithm $\Li_{s}(x)$. Replacing $S_n(s)$ by its integral formula

\begin{align}
S_n(s)=\frac{1}{\Gamma(s)}
\int_{0}^{\infty}(1-e^{-t})^{n-1}e^{-t}t^{s-1}\,dt, \re(s)>0,\label{sec3-eq24}
\end{align}

and summing over $n$, we get

\begin{align}
\phi(s,x)&=\frac{x}{\Gamma(s)}
\int_{0}^{\infty}\frac{e^{-t}}{1-(1-e^{-t})x}t^{s-1}\,dt, \re(s)>0.\label{sec3-eq25}
\end{align}

A change of variable $u=e^{-t}$, a slight rearrangement, and  identification with the integral formula (\ref{sec3-eq4}) defining $\Li_{s}(x)$ leads to the expression 

\begin{align}
\phi(s,x)&=-\frac{1}{\Gamma(s)}
\int_{0}^{1}\frac{(-\log(u))^{s-1}}{\frac{x-1}{x}-u}\,du, \re(s)>0,\nonumber\\
&=-\Li_{s}\left(\frac{x}{x-1}\right).\label{sec3-eq26}
\end{align}

Finally, a combination of the last expression and equation  (\ref{sec3-eq21}) gives the following neat result

\begin{theorem}\label{sec3-thm2}
We have
\begin{align}\label{sec3-eq27}
(s-1)\zeta(s)&=-\int_{0}^{1}\Li_{s}\left(\frac{z}{z-1}\right)\,dz, \end{align}

and generally, within the unit circle
\begin{align}
Z(s,x)&=-\int_{0}^{x}\Li_{s}\left(\frac{z}{z-1}\right)\,dz,  \label{sec3-eq28}\end{align}
where the integrals are interpreted as line integrals.
\end{theorem}

\fancyfoot{}
 \fancyfoot[C]{\leavevmode
 \put(0,0){\color{lightaqua}\circle*{34}}
 \put(0,0){\color{myaqua}\circle{34}}
 \put(-5,-3){\color{myaqua}\thepage}}

Actually, the variable $x$ in the last theorem need not be restricted to the unit circle, and the integrals can be interpreted as curvilinear integrals instead of line integrals as long as we carefully restrict the integrand to a portion of the plane where it is single-valued. For this purpose, let's further analyze the function $\phi(s,x)$.

The integral in (\ref{sec3-eq26}) provides the analytic continuation of $\phi(s,x)$ to the whole $x$-plane save values of $x$ for which the integral is not defined, i.e. $\frac{x-1}{x}\notin \intervalle{0}{1}$; that is,  when $x\notin \intervallefo{1}{+\infty}$. The integral expression  also shows that $\Li_{s}\left(\frac{x}{x-1}\right)$ is a multi-valued function with the points 0, 1 and $\infty$ as the only branch point singularities. The points $1$ and $\infty$ are the only branch points of the principal branch which coincides with the series $\phi(s,x)$. The principal branch is single valued if the plane is cut along $\intervallefo{1}{+\infty}$. The point $x=0$ is however a singular point for all the other branches of $\Li_{s}\left(\frac{x}{x-1}\right)$, and to make the function single valued, the plane should be cut along $\intervalleof{-\infty}{0}\cup \intervallefo{1}{\infty}$. For ease of notation, and from now on, we refer to  the cut plane set  
$\mathbb{C}\setminus
\left\{\intervalleof{-\infty}{0}\cup \intervallefo{1}{\infty}\right\}$ as $\mathbb{C}^{*}$, and we refer to $\phi(s,x)$ by its analytic continuation $\Li_{s}\left(\frac{x}{x-1}\right)$.

The next theorem expresses $Z(s,x)$ in a form which differs from (\ref{sec3-eq28}). The expression of $Z(s,x)$ still uses the polylogarithm but it is rewritten as a sum of a constant term and a variable error term: 

\begin{theorem}\label{sec3-thm3}
For $x\in \mathbb{C}\setminus \{\intervalleoo{-\infty}{0}\cup \intervalleoo{1}{\infty}\}$, we have
\begin{align}\label{sec3-eq29}
Z(s,x)&=(s-1)\zeta(s)+\int_{x}^{1}\Li_{s}\left( \frac{z}{z-1}\right)\,dz,
\end{align}
the integral being taken along a finite continuous path in the $x$-cut plane $\mathbb{C}^{*} \triangleq \mathbb{C}\setminus \{\intervalleof{-\infty}{0}\cup \intervallefo{1}{\infty}\}$, and where the error term satisfies

\begin{align}\label{sec3-eq30}
\int_{x}^{1}\Li_{s}\left( \frac{z}{z-1}\right)\,dz& \sim \frac{-1}{\Gamma(1+s)}(1-x)\left[- \log\left(\frac{1}{x}-1\right)\right]^{s},
\end{align}

as $x\to 1$, $x\in\mathbb{C}\setminus \intervalleoo{1}{\infty}$.
\end{theorem}
\begin{proof}

Again using the series expression $\phi(s,x)$ and staying within the circle of convergence, $\Li_{s}\left(\frac{x}{x-1}\right)$ can be integrated term-by-term

\begin{align}
\int_{x}^{1}\Li_{s}\left(\frac{z}{z-1}\right)\,dz&=-(s-1)\zeta(s)+\sum_{n=1}^{\infty}\frac{S_{n}(s)}{n+1}x^{n+1}\nonumber\\
&=-(s-1)\zeta(s)+Z(s,x),\label{sec3-eq31}
\end{align}

where $x$ is now a variable and the upper limit of integration is a not of 
point of absolute convergence of the series defining 
$\Li_{s}\left(\frac{x}{x-1}\right)$. But since the integrated series converges 
to 0 when $x=1$, Abel's limit theorem on the continuity of power series 
shows that the integral in (\ref{sec3-eq31}) is well-defined. We next show that 
the left hand side of (\ref{sec3-eq31}) defines a global primitive in 
$\mathbb{C}^{*}$.  

We know that a global primitive exists as long as the integral depends only
 on the ordinary endpoints of an admissible path. By an admissible path we 
mean a continuous path  of finite length with ordinary endpoints and such 
that it does not cross the cut $\intervalleof{-\infty}{0}\cup \intervallefo{1}{\infty}$. This is possible since $\mathbb{C}^{*}$ is a star-shaped domain, 
hence it is arc-wise connected; that is, it is always possible that the 
initial and final points of integration remain on the same branch of 
$\Li_{s}\left(\frac{x}{x-1}\right)$. 

But in (\ref{sec3-eq31}), one of the endpoints is singular and so the 
classical theorem on the existence of global primitive needs to be extended 
to integrals with singular endpoints.  This is possible since Cauchy-Goursat theorem and primitive functions can be generalized to improper complex 
integrals provided that the involved integrals are convergent 
\cite[p. 261]{jordan} and \cite[Theorem~$3.3^{\prime}$, p. 89]{lu}. 
Therefore, the integral is well-defined for any $x\in \mathbb{C}^{*}$. It is 
also obviously well-defined for $x=0$ and $x=1$.

To prove the approximation of the error term, we first note that when $x\in \mathbb{C}^{*}$, 
$\log\left(\frac{z}{1-z}\right)=-\log\left(\frac{1-z}{z}\right)$. We then use the estimate of Corollary~\ref{sec3-cor1} which provides 

\begin{align}\label{sec3-eq32}
\Li_{s}\left( \frac{z}{z-1}\right)& \sim \frac{-1}{\Gamma(1+s)}\left[- \log\left(\frac{1}{z}-1\right)\right]^{s},
\end{align}

as $z\to 1$, $z\in \mathbb{C}^{*}$. Finally, since $z$ is in a neighborhood of $1$, we can lift the restriction on $z$ near the point $0$, and the estimate is valid when $z\in\mathbb{C}\setminus \intervalleoo{1}{\infty}$.

To complete the proof, we suppose, without loss of generality, that the integral is taken along a real path, i.e. $x\in \intervalleoo{0}{1}$ since 
$\Li_{s}\left(\frac{z}{z-1}\right)$ being analytic and single-valued in  
$\mathbb{C}\setminus \{\intervalleof{-\infty}{0}\cup \intervallefo{1}{\infty}\}$
Cauchy's theorem permits the deformation of the path to a new real path as 
long as  $x\to 1$. We will show that

\begin{align}\label{sec3-eq33}
&\lim\limits_{x \to 1}\frac{\int_{x}^{1}\Li_{s}\left( \frac{z}{z-1}\right)\,dz}{(1-x)\left[- \log\left(\frac{1}{z}-1\right)\right]^{s}}=
\frac{-1}{\Gamma(s+1)}.
\end{align}

Indeed, let $f(x)=\int_{x}^{1}\Li_{s}\left( \frac{z}{z-1}\right)\,dz$ and let $g(x)=\left(1-x\right)\left[- \log\left(\frac{1}{x}-1\right)\right]^{s}$, where $s=\sigma+i\omega$ is a complex number. It can be easily seen that $\lim\limits_{x\to 1}f(x)=\lim\limits_{x\to 1}g(x)=0$.
In addition, from (\ref{sec3-eq32}), $f^{\prime}(x)$ satisfies the estimate

\begin{align}\label{sec3-eq34}
&f^{\prime}(x)=-\Li_{s}\left( \frac{x}{x-1}\right) \sim \frac{1}{\Gamma(1+s)}\left[- \log\left(\frac{1}{x}-1\right)\right]^{s}, x\to 1.
\end{align}

All we need now is to invoke the following theorem which provides conditions for the applications of  L'H\^{o}pital's rule for complex-valued functions:

\begin{theorem}[L'H\^{o}pital's rule for complex-valued functions \cite{carter}]\label{sec3-thm4}
Let  $f,g : \intervalleoo{c}{1} \longrightarrow \mathbb{C}$ be differentiable on 
the open interval $\intervalleoo{c}{1}$ and let
\[\lim\limits_{x\to 1}f(x)=\lim\limits_{x\to 1}g(x)=0.\]

If the two conditions

\begin{description}
\item[\rm (i)] $\displaystyle{\lim_{x\to 1}\frac{f^{\prime}(x)}{g^{\prime}(x)}=L}$,
\item[\rm (ii)] the derivative $\abs{g(x)}^{\prime}$ of $\abs{g(x)}$ exists and does not vanish in 
$\intervalleoo{c}{1}$, and  the ratio $\frac{\abs{g^{\prime}(x)}}{\abs{g(x)}^{\prime}}$ is bounded on
$\intervalleoo{c}{1}$,
\end{description}
then

\[\lim_{x\to 1}\frac{f(x)}{g(x)}=L.\]

\end{theorem}

It remains to verify the conditions of  Theorem~\ref{sec3-thm4}. Clearly, with $L=\frac{-1}{\Gamma(s+1)}$ and  (\ref{sec3-eq34}), condition {\rm (i)} is verified. And, it is easy to verify that the function $g(x)$ verifies condition {\rm (ii)}. Indeed, when $x\in \intervalleoo{c}{1}$ for some $c$ close to 1, we have

\begin{align}\label{sec3-eq35}
\abs{g(x)}&=\left(1-x\right)\left[- \log\left(\frac{1}{x}-1\right)\right]^{\sigma},\\
g^{\prime}(x)&=\left[- \log\left(\frac{1}{x}-1\right)\right]^{s}\left[-1- \frac{s}{x\log\left(\frac{1}{x}-1\right)}  \right],\label{sec3-eq36}\\
\abs{g(x)}^{\prime}&=\left[- \log\left(\frac{1}{x}-1\right)\right]^{\sigma}\left[-1- \frac{\sigma}{x\log\left(\frac{1}{x}-1\right)}   \right]\label{sec3-eq37}
\end{align}

By choosing $c$ very close to 1, the terms $\frac{s}{x^2\log\left(\frac{1}{x}-1\right)}$ and $\frac{\sigma}{x^2\log\left(\frac{1}{x}-1\right)}$ in (\ref{sec3-eq36})-(\ref{sec3-eq37}) can be made arbitrarily close to 0, and therefore, the ratio $\frac{\abs{g^{\prime}(x)}}{\abs{g(x)}^{\prime}}$ can be made arbitrarily close to 1. The proof of the theorem follows.

\end{proof}

For a trivial or nontrivial zero $s$, two  direct consequences of 
Theorem~\ref{sec3-thm3} are:  (i) $\displaystyle{\int_{0}^{1}\Li_{s}\left(\frac{z}{z-1}\right)\,dz=0}$ and \newline  
(ii) $\displaystyle{\int_{0}^{\frac{1}{2}}\Li_{s}\left(\frac{z}{z-1}\right)\,dz= 
-\int_{\frac{1}{2}}^{1}\Li_{s}\left(\frac{z}{z-1}\right)\,dz}$, where  the
 special point $z=\frac{1}{2}$ splits the integral into two equal parts 
in absolute value and annuls the integrand since 
$\displaystyle{\Li_{s}\left(\frac{\frac{1}{2}}{\frac{1}{2}-1}\right)=\Li_{s}(-1)=0}$. For negative integer values of $s$, $\Li_{s}\left(\frac{z}{z-1}\right)$ 
are polynomials as we will see in the next section, and the geometry of 
their zeros and identities (i) and (ii) can be easily analyzed. When $s$ is not a negative integer, the zeros of $\Li_{s}\left(\frac{z}{z-1}\right)$ and its integral  should be finite in number according to \cite{gawronski} since they do not accumulate at the critical points $z=1$ or $z=\infty$.

\subsection{The Parametrized Series at Integer Values of \texorpdfstring{$s$}{s}}\label{subsec3_3}

\noindent When $s=-k\in \mathbb{Z}^{-}$, $S(n,-k)$ are zero for $n> k+1$. The series 
$\phi(s,x)$ in (\ref{sec3-eq20})
and the globally convergent series $Z(s,x)$ in (\ref{sec3-eq19}) become polynomials of degree 
$k+1$ and $k+2$ respectively: 

\begin{align}
-\phi(-k,x)&=\Li_{-k}\left(\frac{z}{z-1}\right)\nonumber\\
&= (-1)^{k+1}z(z-1)^{k}A_{k}\left(\frac{z}{z-1}\right)\nonumber\\
&=(-1)^{k+1}\sum_{j=0}^{k-1}\eulerian{k}{j}z^{j+1}(z-1)^{k-j+1},\label{sec3-eq38}\\
Z(-k,x)&=(-1)^{k}\int_{0}^{x}\sum_{j=0}^{k-1}\eulerian{k}{j}z^{j+1}(z-1)^{k-j+1}\,dz,\label{sec3-eq39}
\end{align}

where $A_k(z)=\sum_{j=0}^{k-1}\eulerian{k}{j}z^{j}$ are the Eulerian polynomials and ${\eulerian{k}{j}}$ are the Eulerian numbers. When $s=k\in \mathbb{Z}^{+}$, the Lerch-Jonqui{\`e}re formulas in Theorem~\ref{sec3-thm1} are reduced to finite terms:

\begin{enumerate}
\item  For complex $z\notin \intervalleof{-\infty}{0}$:

\begin{align}
\Li_{k}\left(\frac{z}{z-1}\right)+(-1)^{k}\Li_{k}\left(\frac{z-1}{z}\right)&=-\frac{(2\pi i)^{k}}{k!}B_{k}\left(\frac{1}{2}+\frac{\ln\left(\frac{z}{1-z}\right)}{2\pi i}\right),\label{sec3-eq40}
\end{align}

\item  For complex $z\notin \intervallefo{1}{\infty}$:
\begin{align}
\Li_{k}\left(\frac{z}{z-1}\right)+(-1)^{k}\Li_{k}\left(\frac{z-1}{z}\right)&=-\frac{(2\pi i)^{k}}{k!}B_{k}\left(\frac{1}{2}-\frac{\ln\left(\frac{1-z}{z}\right)}{2\pi i}\right),\label{sec3-eq41}
\end{align}

\item For $z\notin
\left\{\intervalleof{-\infty}{0}\cup \intervallefo{1}{\infty}\right\}$:
\begin{align}
\Li_{k}\left(\frac{z}{z-1}\right)+(-1)^{k}\Li_{k}\left(\frac{z-1}{z}\right)&=-\frac{(2\pi i)^{k}}{k!}B_{k}\left(\frac{\ln\left(\frac{z}{z-1}\right)}{2\pi i}\right),\label{sec3-eq42}
\end{align}
\end{enumerate}

where $B_{k}(x)$ are the Bernoulli polynomials of order $k$. Thus, for example around $z=1$, we can use (\ref{sec3-eq41}) to find a local analytic expression for 
$Z(k,x)$. We find

\begin{align}
\Li_{k}\left(\frac{z}{z-1}\right)&=-\frac{(2\pi i)^{k}}{k!}B_{k}\left(\frac{1}{2}-\frac{\ln\left(\frac{1-z}{z}\right)}{2\pi i}\right)+\Omega(z)\nonumber\\
&=\text{albegraic polynomial in~} \ln(1-z), \ln(z) +\Omega(z),\label{sec3-eq43}
\end{align}

where $\Omega(z)=(-1)^{k}\Li_{k}\left(\frac{z-1}{z}\right)$ is regular around 
$z=1$. And if we use the explicit well-known expression 
$B_k(x)=\sum_{j=0}^{k}\binom{k}{j}
 B_{k-j}x^{j}$, where $B_{k-j}$ are the Bernoulli numbers, we get

\begin{align}
Z(k,x)&=\frac{(2\pi i)^{k}}{k!}\sum_{j=0}^{k}\binom{k}{j} B_{k-j}
\int_{0}^{x} \left(\frac{1}{2}-\frac{\ln\left(\frac{1-z}{z}\right)}{2\pi i}\right)^{j} \,dz+\int_{0}^{x} \Omega(z) \,dz\nonumber\\
&=\text{albegraic polynomial in~} (1-x), \ln(1-x)+\Omega_1(x),\label{sec3-eq44}
\end{align}

where $\Omega_1(x)$ is analytic at $x=1$. In the following example, we illustrate use of this formula by expressing the dilogarithm  function $\Li_{2}(x)$.
\begin{example}\label{sec4-example1} For $k=2$, $B_2(x)=x^2-x+\frac{1}{6}$. We obtain the formulas

\begin{align}
\Li_{2}\left(\frac{z}{z-1}\right)+\Li_{2}\left(\frac{z-1}{z}\right)&=2\pi^{2}B_{2}\left(\frac{1}{2}-\frac{\ln\left(\frac{1-z}{z}\right)}{2\pi i}\right),
\label{sec3-eq45}\\
B_2\left(\frac{1}{2}-\frac{\ln\left(\frac{1-z}{z}\right)}{2\pi i}\right)&=
-\frac{1}{12}-\frac{1}{4\pi^2}\ln\left(\frac{1-z}{z}\right)^2\nonumber\\
&=-\frac{\ln(1-z)^2}{2}+\ln(1-z)\ln(z)-\frac{\ln(z)^2}{2}-\frac{\pi^2}{6}.\label{sec3-eq46}
\end{align}
Accordingly,
\begin{align}
\Li_{2}\left(\frac{z}{z-1}\right)&=-\frac{\ln(1-z)^2}{2}+\Li_2(x),\label{sec3-eq47}\\
Z(2,x)&=\zeta(2)+(1-x)\bigg[\frac{1}{2}\ln^{2}(1-x)+\ln(1-x)-1\bigg]-\int_{x}^{1}\Li_{2}(x)\,dx.\label{sec3-eq48}
\end{align}
\end{example}

A summary of the formulas for $Z(s,x)$ for different values of $s$ is illustrated in Table~\ref{sec3-table1}.

\vspace{1mm}

\begin{table}[!ht]
\small
\centering
\begin{tabular}{BB} 
\toprule
\multicolumn{4}{l}{\bf Negative integer values of $s$} \\
\midrule
\Li_{-k}\left(\frac{z}{z-1}\right) &=(-1)^{k+1}\sum_{j=0}^{k-1}\eulerian{k}{j}z^{j+1}(z-1)^{k-j+1} & Z(-k,x) &= (-1)^{k}\int_{0}^{x}\sum_{j=0}^{k-1}\eulerian{k}{j}z^{j+1}(z-1)^{k-j+1}\,dz \\
 \vdots~~~~ & ~~~~~~~\vdots& \vdots~~~~ &=~~~~ \vdots \\
\Li_{-4}\left(\frac{z}{z-1}\right) &=z(z-1)(2z-1)(12z^2-12z+1) & Z(-4,x) &= \frac{x^2}{2}(x-1)^2(8x^2-8x+1) \\
\Li_{-3}\left(\frac{z}{z-1}\right) &=-z(z-1)(6z^2-6z+1) & Z(-3,x) &= \frac{x^2}{2}\left(\frac{12}{5}x^3-6x^2+\frac{14}{3}x-1\right) \\
\Li_{-2}\left(\frac{z}{z-1}\right)    &=-z(z-1)(2z-1) & Z(-2,x)  &=\frac{x^2}{2}(x-1)^2  \\
\Li_{-1}\left(\frac{z}{z-1}\right)    &=z(z-1) & Z(-1,x)  &= \frac{x^2}{2}\left(-\frac{2}{3}x+1\right) \\
\Li_{0}\left(\frac{z}{z-1}\right)    &= -z & Z(0,x)    &= \frac{x^2}{2} \\
\midrule
\multicolumn{4}{l}{\bf Critical Strip $0<\re(s)<1$} \\
\midrule
\Li_{s}\left(\frac{z}{z-1}\right)   &\sim -\frac{(-\ln(\frac{1}{z}-1))^{s}}{\Gamma(s+1)},z\to 1, z\in \mathbb{C}^{*}& Z(s,x)      &= (s-1)\zeta(s)+(x-1)\frac{(-\ln(\frac{1}{x}-1))^{s}}{\Gamma(s+1)}[1+\cdots] \\
\midrule
\multicolumn{4}{l}{\bf Positive integer values of $s$} \\
\midrule
\Li_{1}\left(\frac{z}{z-1}\right)    &= \ln(\frac{1}{z}-1)+\ln(z)\sim \ln(1-z), z\to 1, z\in \mathbb{C}^{*}   & Z(1,x)   &= 1+(x-1)\left[1-\ln(1-x)\right] \\
\Li_{2}\left(\frac{z}{z-1}\right)    &= -\frac{1}{2}\ln^{2}(\frac{1}{z}-1) -\frac{\pi^2}{6}-\Li_{2}(1-\frac{1}{z}) & Z(2,x)     &= \zeta(2)+(x-1)\bigg[\frac{1}{2}\ln^{2}(1-x)\\
&\sim -\frac{1}{2}\ln^{2}(1-z), z\to 1, z\in \mathbb{C}^{*} &  &\qquad +\ln(1-x)-1\bigg]-\int_{x}^{1}\Li_{2}(x)\,dx  \\
\vdots~~~~ & ~~~~~~~~\vdots& \vdots~~~~ &=~~~~ \vdots \\
\Li_{n}\left(\frac{z}{z-1}\right)    &= -\frac{(2\pi i)^n}{n!}B_n\left(\frac{1}{2}-\frac{\ln(\frac{1}{z}-1)}{2\pi i}\right)-\Li_{n}\left(1-\frac{1}{z}\right)  & Z(n,x)     &= (n-1)\zeta(n)+(x-1)\frac{\ln^{n}(1-x)}{n!}[1+\cdots] \\
&\sim (-1)^{n+1}\frac{1}{n!}\ln^{n}(1-z), z\to 1, z\in \mathbb{C}^{*} &  &  \\
\bottomrule
\end{tabular}
\caption{Table of the parametrized zeta function $Z(s,x)$ for different values of $s$; $A_n$ are Eulerian polynomials, ${\eulerian{k}{j}}$ the Eulerian numbers, 
 and $\mathbb{C}^{*}=\mathbb{C}\setminus
\left\{\intervalleof{-\infty}{0}\cup \intervallefo{1}{\infty}\right\}$.
}\label{sec3-table1}
\end{table}

\section{Qualitative Characterization of the Nontrivial Zeros of \texorpdfstring{$\zeta(s)$}{zeta(s)}}\label{sec4}
 
\noindent
 The  zeros of $\zeta(s)$ come into two types. The trivial zeros which occur at all negative even integers
$s=-2,-4,\cdots$, and the nontrivial zeros which occur at certain
values of $s\in \mathbb{C}$ in the critical strip $0<\re(s)<1$ and whose existence in infinite number has been know for a long time. Since 
$\zeta(s)$ is an analytic function, the nontrivial
zeros must be isolated and each must have a finite multiplicity $m$. The multiplicity $m$ of
the non-trivial zero is believed by many experts to be 1 but
remains to this date unproved.
Thus, if $s_0$ is a non-trivial zero, there exists a neighborhood
$B_{\delta}(s_0)$, $\delta>0$ sufficiently small, such that $s_0$
is the unique zero in $B_{\delta}(s_0)$. 

Let $s_0$ be one of the nontrivial zeros. Let the function
$\mathscr{R}(s)$ be defined for all $s\in B_{\delta}(s_0)$,
$s\ne s_0$ by:

\begin{equation}\label{sec4-eq1}
\mathscr{R}(s)\triangleq\frac{\zeta(s)}{\zeta(1-s)}, s\ne s_0.
\end{equation}

The functional equation of the Riemann zeta
function implies the existence of the following limit:

\begin{equation}\label{sec4-eq2}
\lim\limits_{s \to s_0}\frac{\zeta(s)}{\zeta(1-s)}=
\frac{\pi^{-\frac{1-s_0}{2}}\Gamma(\frac{1-s_0}{2})}{\pi^{-\frac{s_0}{2}}
\Gamma(\frac{s_0}{2})}.
\end{equation}

Consequently, we can extend by continuity the definition of $\mathscr{R}(s)$ to
the whole set $B_{\delta}(s_0)$. The new extension by continuity,
also denoted by $\mathscr{R}(s)$, is given by:

\begin{equation}\label{sec4-eq3}
    \mathscr{R}(s)\triangleq
\begin{cases} \frac{\zeta(s)}{\zeta(1-s)} &\mbox{if } s \ne s_0 \\
\frac{\pi^{-\frac{1-s_0}{2}}\Gamma(\frac{1-s_0}{2})}{\pi^{-\frac{s_0}{2}}
\Gamma(\frac{s_0}{2})}& \mbox{if } s=s_0.
\end{cases}
\end{equation}

This implies that for a non-trivial zero $s_0$, the ratio

\begin{equation}\label{sec4-eq4}
\frac{\zeta(s_0)}{\zeta(1-s_0)} =\mathscr{R}(s_0)
\end{equation}

is well-defined and is a finite non-zero number. In a similar fashion, we can also define a ratio of parametrized zeta functions using the parametrized zeta function $Z(s,x)$ defined by (\ref{sec3-eq19}) in section~\ref{sec3}. This is described next.

\subsection{Ratio of Parametrized Zeta Series}\label{subsec4_1}

 \noindent

In this section, we will purposely dwell on the details because interchanging limits of functions of two variables can be very delicate. Recall the Abel parametrization $Z : B_{\delta}(s_0)\times (0,1) \longrightarrow
\mathbb{C}$ of the uniformly absolutely-convergent
series for $\zeta(s)$ defined by

\begin{align}
Z(s,x)&=\sum_{n=1}^{\infty}\frac{S_{n}(s)}{n+1}x^{n+1},\label{sec4-eq5}
\end{align}

where  $B_{\delta}(s_0)$ is an open ball containing the only nontrivial zero $s_0$. For fixed $s$, uniform-absolute convergence with respect to $x$ implies continuity at
the boundary point $x=1$:

\begin{eqnarray}  \label{sec4-eq6}
\lim\limits_{x \to 1}Z(s, x)&=&Z(s, 1)=\sum_{n=1}^{\infty}\frac{S_{n}(s)}{n+1}=(s-1)\zeta(s).
\end{eqnarray}

Moreover, by uniform-absolute convergence with respect to $s$ the
existence of the following repeated limit is immediate:

\begin{eqnarray}  \label{sec4-eq7}
\lim\limits_{s \to s_0}\lim\limits_{x \to 1}Z(s,
x)&=&\lim\limits_{s \to s_0}\sum_{n=1}^{\infty}\frac{S_{n}(s)}{n+1}= 
(s_0-1)\zeta(s_0).
\end{eqnarray}

Similarly,

\begin{eqnarray}  \label{sec4-eq8}
\lim\limits_{x \to 1}\lim\limits_{s \to s_0}Z(s,
x)&=&\lim\limits_{x \to 1}\sum_{n=1}^{\infty}\frac{S_{n}(s_0)}{n+1}x^{n+1}=\sum_{n=1}^{\infty}\frac{S_{n}(s_0)}{n+1}=(s_0-1)\zeta(s_0).
\end{eqnarray}

Therefore, from the parametrized series definition and for a
non-trivial zero $s_0$, we always have

\begin{align}  \label{sec4-eq9}
\lim\limits_{s \to s_0}\lim\limits_{x \to 1}Z(s,
x)&=\lim\limits_{x \to 1}\lim\limits_{s \to s_0}Z(s,
x)=(s_0-1)\zeta(s_0)=0.
\end{align}

The same identity  holds for $1-s_0$ which is also a zero of
$(s-1)\zeta(s)$. From these found identities, we now deduce a
similar identity for the ratio of zeta functions. Indeed, by the
definition of limit of the ratio of two functions, we know that

\begin{eqnarray}
  \lim\limits_{x \to 1}\frac{Z(s, x)}{Z(1-s, x)}&=& \frac{\lim\limits_{x \to 1}Z(s, x)}
  {\lim\limits_{x \to 1}Z(1-s, x)}=\frac{(s-1)\zeta(s)}{-s\zeta(1-s)}
  \label{sec4-eq10}
\end{eqnarray}

\emph{provided the right hand side of (\ref{sec4-eq10}) exists}.
For $s\ne s_0$, the right hand side obviously exists and is a
finite nonzero number. Moreover, from the discussion in the beginning of this section, the right hand side exists of (\ref{sec4-eq10})
even for $s=s_0$ which is originally an indeterminate form and
made determined using the functional equation. Therefore, with the
absolute value being continuous, we get

\begin{eqnarray}
  \lim\limits_{s \to s_0}\lim\limits_{x \to 1}\left|\frac{Z(s, x)}{Z(1-s, x)}\right|&=&
  \lim\limits_{s \to
  s_0}\left|\frac{(s-1)\zeta(s)}{-s\zeta(1-s)}\right|=\left|\frac{s_0-1}{-s_0}\mathscr{R}(s_0)\right|.
  \label{sec4-eq11}
\end{eqnarray}

Similarly, continuity of $Z(s, x)$ with respect to $s$ gives

\begin{eqnarray}
 \lim\limits_{s \to s_0}\frac{Z(s, x)}{Z(1-s, x)}&=&
  \frac{Z(s_0, x)}{Z(1-s_0, x)}.
  \label{sec4-eq12}
\end{eqnarray}

Hence,

\begin{eqnarray}
  \lim\limits_{x \to 1}\lim\limits_{s \to s_0}\left|\frac{Z(s, x)}{Z(1-s, x)}\right|&=&
  \frac{\lim\limits_{x \to
  1}\left|Z(s_0, x)\right|}{\lim\limits_{x \to
  1}\left|Z(1-s_0, x)\right|}.
  \label{sec4-eq13}
\end{eqnarray}

\emph{provided the right hand side exists}. Indeed, both the
numerator and the denominator exist because of (\ref{sec4-eq8}), 
and as a consequence of the functional equation, the ratio
also exists.

In other words, for the parametrized uniformly absolutely 
convergent series of $Z(s, x)$ defined by (\ref{sec4-eq5}) and for any
nontrivial zero $s_0$, we have the identity:

\begin{eqnarray}
  \lim\limits_{s \to s_0}\lim\limits_{x \to 1}\left|\frac{Z(s, x)}{Z(1-s, x)}\right|&=& \lim\limits_{x \to 1}\lim\limits_{s \to s_0}\left|\frac{Z(s, x)}{Z(1-s,
  x)}\right|=\left|\frac{s_0-1}{-s_0}\mathscr{R}(s_0)\right|.
  \label{sec4-eq14}
\end{eqnarray}

This identity asserts the existence and finiteness of the repeated
limits \emph{provided all the involved limits exist}. The
existence of the involved limits is the main reason for
considering the absolute value of the quotient. Without absolute
values, the limits may not exist for some particular values of $s$
because the ratio is a complex-valued function.

Identity (\ref{sec4-eq14}) should also hold if we rewrite $Z(s, x)$ using a
different analytic expression. Indeed, uniform absolute-convergence implies uniform convergence, and absolute-convergence is independent of the
ordering of a series. The parametrized series will  remain
uniformly convergent no matter how the order of the terms are changed or rearranged. Other expressions can for example be obtained by rearranging terms, power series expansion around a regular point in the domain of analyticity, the sum of a regular part and a multi-valued (or principal) part, expansion in powers of
$\log{x}$ etc.  In the next section, we will make use of this fact using the analytic expression (\ref{sec3-eq30}) given in Theorem~\ref{sec3-thm3}.

\subsection{The Riemann Hypothesis}\label{subsec4_2}

\noindent Let $s_0=\sigma_0+i\omega_0$, $0<\re(s_0)<1$ be a fixed nontrivial
zero of $\zeta(s)$ and consider the parametrized zeta function $Z(s_0,x)$ defined by the series (\ref{sec4-eq5}).  Theorem~\ref{sec3-thm3} permits to express the series as

\begin{align}\label{sec4-eq15}
Z(s_0,x)&=\int_{x}^{1}\Li_{s_0}\left( \frac{z}{z-1}\right)\,dz,
\end{align}

since $(s_0-1)\zeta(s_0)=0$. Replacing the last expression and the corresponding one for the zero $1-s_0$ into the middle of equation (\ref{sec4-eq14}) which asserts the existence and
finiteness of the repeated limits, we get

\begin{align}\label{sec4-eq16}
\lim\limits_{x \to 1}\frac{\abs{\int_{x}^{1}\Li_{s_0}\left( \frac{z}{z-1}\right)\,dz}}
{\abs{\int_{x}^{1}\Li_{1-s_0}\left( \frac{z}{z-1}\right)\,dz}}&=\abs{
\mathscr{R}(s_0)\frac{s_0-1}{-s_0}}.
\end{align}

Now, using the asymptotic estimate (\ref{sec3-eq30}) of Theorem~\ref{sec3-thm3}, the last limit can rewritten as

\begin{align}\label{sec4-eq17}
\abs{\frac{(1-s_0)\Gamma(1-s_0)}{
s_0\Gamma(s_0)}}\lim\limits_{x
\to 1}\left[- \log\left(\frac{1}{x}-1\right)\right]^{2\sigma_0-1}=\abs{
\mathscr{R}(s_0)\frac{s_0-1}{-s_0}}.
\end{align}

This limit always exists but its value depends of $\sigma_0$:

\begin{equation}\label{sec4-eq18}
    \lim\limits_{x
\to 1}\left[- \log\left(\frac{1}{x}-1\right)\right]^{2\sigma_0-1} =
\begin{cases} 0 &\mbox{if } \sigma_0 < \frac{1}{2} \\
1 & \mbox{if } \sigma_0 = \frac{1}{2}\\
\infty & \mbox{if } \sigma_0 > \frac{1}{2}.
\end{cases}
\end{equation}

However, for the limit to be consistent with  (\ref{sec4-eq17}) even if we switch the role of $\zeta(s)$ and $\zeta(1-s)$ in the ratio, $\sigma_0$ has to be equal to $\frac{1}{2}$. This is nothing but the Riemann Hypothesis.

\section{What About Other Zeta Series and Parametrizations?}\label{sec5}
 \noindent

In this section we examine whether any parametrization of the zeta
function leads to the same results. The  elementary argument used in the previous section is uniquely based on the fact that the ratio
$\tfrac{\zeta(s)}{\zeta(1-s)}=\mathscr{R}(s)$ is a continuous
function of $s$ and that it has a finite value at any non-trivial
zero $s_0$. Therefore, the reasoning  should be independent of the
parametrization $\zeta(s)$ as long as the  parametrized
zeta function $Z(s,x)$ possess some specific properties. The properties of $Z(s,x)$ that were used in the proof in the previous section are the following:

\begin{description}
\item[A1] $Z(s,x)$ is a uniformly absolutely convergent parametrization of an absolutely uniformly convergent (in the critical strip) series  $(s-1)\zeta(s)$ of the form
\begin{align}
(s-1)\zeta(s)=&\sum_{n=1}^{\infty}u(n,s),\label{sec5-eq1}
\end{align}
  where $u(n,s)$  is a well-defined
holomorphic function whose restriction to the real numbers is
real.

\item[A2] $x=1$ is the unique finite singular point of $Z(s,x)$.
\end{description}

Assumption (A1) regarding uniform absolute-convergence is fundamental if we
wish to conserve at least pointwise and uniform convergence since any
non-absolutely convergent series can potentially be reordered into a
non-uniformly convergent series, or a series which does not even
converge pointwise. In addition,  with (A1) $\lim\limits_{x
\to 1}Z(s,x)=(s-1)\zeta(s)$ and $u(n,\overline{s}) = \overline{u(n,s)}$ so that $(s-1)\zeta(s)$ is real when $s$ is real.

Assumption (A2) is less fundamental but requiring it will eliminate series that are regular at $x=1$ such as Knopp's series above. When $x=1$ is not singular, uninteresting conclusions will arise as we will see in another example at the end of this section (see Example~\ref{sec5-example2}). To simplify the analysis and to avoid unnecessary generalizations, we continue with  Abel type  parametrizations only:

\begin{align}
Z(s,x)=&\sum_{n=1}^{\infty}u(n,s)x^{n},\label{sec5-eq2}
\end{align}

With such parametrizations, assumption (A1) can be easily checked and Assumption (A2) can for example be satisfied using the general sufficient growth conditions considered  by Lindel{\"o}f in \cite[pp. 132]{lindelof:1905}. In addition, with 
(A2) and for fixed $s$, we must have
\begin{align}
\lim\limits_{n
\to \infty}\frac{u(n,s)}{u(n+1,s)}=1.\label{sec5-eq3}
\end{align}

There exists  other series than the series subject of this paper that analytically continue the original series (\ref{sec1-eq1}) to the whole complex plane. A summary of the relevant formulas associated with these series is illustrated in Table~\ref{sec5-table1}. Note that the Dirichlet eta function $\eta(s)=(1-2^{1-s})\zeta(s)=\sum_{n=1}^{\infty}\frac{(-1)^{n-1}}{n^s}$
is not mentioned since it is not an absolutely convergent series when $\re(s)<1$. Note also that the series in the  third row is very recent \cite{blagouchine}, and in the same paper, the author also proves that Hasse's and Ser's series are a rearrangement of each other.

All the  series of Table~\ref{sec5-table1} are defined for all $s\in \mathbb{C}$ and  are uniformly and absolutely convergent series on compact sets of the $s$ plane. Moreover, the rate of convergence of the  different series  can be roughly estimated using the asymptotic order of growth of each particular term of the series, denoted by $u_n$ using the asymptotic estimate of $S_n(s)$ proved in 
Section~\ref{sec2} and the estimates ${G_{n}\thicksim \frac{1}{n(\log{n})^2}}$ and ${C_{n}\thicksim \frac{1}{\log{n}}}$. The numbers $G_n, C_n$ are known respectively as Gregory coefficients\footnote{Gregory coefficients can  also be expressed as $G_n=\frac{1}{n!}\int_{0}^{1}t(1-t)\cdots(n-1-t)\,dt$.}  and Cauchy numbers of the second kind. The asymptotic estimates of these numbers can be easily deduced from the following generating functions of the $G_n$ and $C_n$ using the methods of \cite{flajolet:singularities}:
\begin{align}\label{sec5-eq4}
\frac{1}{z}+\frac{1}{\log(1-z)}&=\sum_{n=1}^{\infty}G_nz^{n},\\
-\frac{1}{z}-\frac{1}{(1-z)\log(1-z)}&=\sum_{n=1}^{\infty}C_nz^{n-1}.\label{sec5-eq5}
\end{align}

\vspace{1mm}

\begin{table}[!ht]
\centering
\begin{tabular}{lcllll} 
\toprule
{\bf Name }  & {\bf Series } & && {\bf Convergence  }   & {\bf Point }\\
             &               && &{\bf Rate  }  &{\bf  $x=1$}\\
\midrule
Knopp\cite{sondow}&$\displaystyle{(1-2^{1-s})\zeta(s)}$&=$\displaystyle{\sum_{n=0}^{\infty}\frac{1}{2^{n+1}}S_{n+1}(s)}$&${u_{n}\thicksim}$&
${ \frac{(\log{n})^{s-1}}{2^{n}n\Gamma(s)}}$&regular\\
Ser\cite{ser}&$\displaystyle{\zeta(s)}$&=$\displaystyle{\frac{1}{s-1}+\sum_{n=1}^{\infty}G_n S_n(s)}$&${u_{n}\thicksim }$&${\frac{(\log{n})^{s-3}}{n^2\Gamma(s)}}$&singular\\
Hasse\cite{hasse}&$\displaystyle{(s-1)\zeta(s)}$&=$\displaystyle{\sum_{n=0}^{\infty}\frac{1}{n+1}S_{n+1}(s-1)}$&&&\\
&&=$\displaystyle{\sum_{n=1}^{\infty}\frac{1}{n}S_{n}(s-1)}$&${u_{n}\thicksim}$&${ \frac{(\log{n})^{s-2}}{n^2\Gamma(s-1)}} $&singular\\
Blagouchine\cite{blagouchine}&$\displaystyle{\zeta(s)}$&=$\displaystyle{\frac{s}{s-1}-\sum_{n=1}^{\infty}C_n\tilde{S}_{n}(s)}$&${u_{n}\thicksim}$&${ \frac{(\log{n})^{s-2}}{n^2\Gamma(s)}} $&singular\\
This paper\cite{lazhar:zeros}&$\displaystyle{(s-1)\zeta(s)}$&=$\displaystyle{\sum_{n=1}^{\infty}\frac{1}{n+1}S_{n}(s)}$&&\\
&&=$\displaystyle{-\sum_{n=1}^{\infty}\frac{1}{n(n+1)}\Delta_{n}(s)}$&
${u_{n}\thicksim}$&${ \frac{(\log{n})^{s-1}}{n(n+1)\Gamma(s)}}$&singular\\
\bottomrule
\end{tabular}
\caption{Abel Parametrization $\sum_{n=0}^{\infty}u(n,s)x^{n}$ of Different Series of $\zeta(s)$. $S_n(s)=\sum_{k=0}^{n-1}(-1)^{k}\binom{n-1}{k}(k+1)^{-s}$, $\tilde{S}_{n}(s)=\sum_{k=0}^{n-1}(-1)^{k}\binom{n-1}{k}(k+2)^{-s}$, $\Delta_n(s)=\sum_{k=1}^{n}\binom{n}{k}(-1)^{k}k^{1-s}$and $u_n\triangleq u(n,s)$.}
\label{sec5-table1}
\end{table}

When $s$ is in the critical strip, the fastest converging series is Knopp's series,  followed by Ser's
series, followed by both hasse and Blagouchine's series which possess equal convergence rates, and followed lastly by this paper's series. Apart from  Knopp's series whose Abel parametrization  violates assumption (A2), the point $x=1$ is a singular point for all the remaining Abel parametrizations. 
The high convergence rate of Ser's series, Hasse's series and Blagouchine's series did not allow the use the method of \cite{lazhar:zeros} which is essentially based on the fact that the first derivative with respect to the parameter $x$ of the parametrized series diverges at the critical point $x=1$. That is, although not advantageous computationally, the slow convergence rate of this paper's series was in some respects very handy in \cite{lazhar:zeros}.

Let's now suppose  that we have managed to express our parametrized zeta function (\ref{sec5-eq2}) in a neighborhood of the singular point $x=1$ by

\begin{align}\label{sec5-eq6}
Z(s,x)=(s-1)\zeta(s)+a(s)(x-1)^{m(s)}[1+\cdots],
\end{align}

where $a(s), m(s)$ are continuous functions of $s$, being real for
$s$ real\footnote{Note that the reasoning does not only concern
expansions of the form (\ref{sec5-eq6}) but also general
expansions that contain algebraic and a finite number of nested
logarithmic factors of the form
\begin{align*}
&Z(s,x)=(s-1)\zeta(s)+
a(s)(1-x)^{m_0(s)}\left(\log\frac{1}{1-x}\right)^{m_1(s)}\cdots
\left(\log\log\ldots\log\frac{1}{1-x}\right)^{m_p(s)}[1+\cdots].
\end{align*}
}. Using the same arguments of the previous section, we
conclude that for a nontrivial zero $s_0$, we must have

\begin{eqnarray}\label{sec5-eq7}
\left |\frac{a(s_0)}{a(1-s_0)}\right|&=&\left |\frac{s_0-1}{-s_0}\mathscr{R}(s_0)\right| \text{and}\\
\re(m(s_0))&=&\re(m(1-s_0)).\label{sec5-eq8}
\end{eqnarray}

In other words, in addition to the obvious condition $0<\re(s_0)<1$, the
functional equation imposes the two extra necessary conditions
(\ref{sec5-eq7})-(\ref{sec5-eq8}). Using the particular parametrization of  the previous section, the solution set  of the two extra 
conditions was the critical line $\re(s_0)=\frac{1}{2}$ which translates into the Riemann Hypothesis. 

To solve (\ref{sec5-eq7})-(\ref{sec5-eq8}) for a general parametrization, we start by narrowing the solution set. Firstly, note that if  $m(s)$ (resp. $a(s)$) is a well-defined
holomorphic function whose restriction to the real numbers is
real-valued, then $m(\overline{s}) = \overline{m(s)}$ (resp.
$a(\overline{s}) = \overline{a(s)}$). Secondly, we  know that for all
$s_0$ on the critical line $\overline{s_0}=1-s_0$; thus, the
critical line is always part of the solution set defined by equations
(\ref{sec5-eq7})-(\ref{sec5-eq8}). This leads to the following question:

\noindent{\bf Question:} Do parametrized zeta functions satisfying (A1)-(A2) exist such that the whole critical line
is only part of a larger solution set imposed by conditions (\ref{sec5-eq7})-(\ref{sec5-eq8})?

We cannot answer such a question in its generality but we know that the answer is  negative if we initially start with a valid zeta series such as the one described in this paper. However, we can find artificially constructed parametrized zeta functions (not zeta series) which provide a positive answer to the above question. We believe that such constructions\footnote{Note that all the artificial constructions have to satisfy the equality on the interchange of repeated limits.} are not admissible in the sense that they do not originate from a globally (or at least in the critical strip) convergent series of $\zeta(s)$. We finish this section with three examples. The first example depicts an artificially constructed parametrization, the second example illustrates a parametrization  that violates assumption (A2), and the third example reproduces the parametrization of \cite{lazhar:zeros}.

\begin{example}\label{sec5-example1}
In this example, we have an artificial parametrization such that
(\ref{sec5-eq7})-(\ref{sec5-eq8}) have the critical line as a solution and other additional solutions off the critical line. Consider the following parametrized perturbation of
$\zeta(s)$:
\begin{align}\label{sec5-eq9}
Z(s,x)=(s-1)\zeta(s)+(x-1)^{s^3}.
\end{align}

Equation (\ref{sec4-eq1}) is trivially satisfied, and with
 $s_0=\sigma_0+i\omega_0$, the solution to
\begin{align}\label{sec5-eq10}
&\re(s_0^3)=\re((1-s_0)^3)
\end{align}

is the union of the vertical line $\sigma_0=\frac{1}{2}$ and the
curve in the $(\sigma_0,\omega_0)$-plane defined by

\begin{align}\label{sec5-eq11}
& 3\omega_0^2=\sigma_0^2-\sigma_0+1.
\end{align}

The curve is a vertical hyperbola centered at $(\frac{1}{2},0)$.
Such a parametrization contains an infinite
number of points off the critical line none of which are solutions
to $\zeta(s)=0$.

In appearance this parametrization is useful. It tells  that the
non-trivial zeros of $\zeta(s)$ have to be either on the hyperbola
or on the critical line. It reduced the solution set  of
$\zeta(s)=0$ from a vertical strip to a union of two segments of a
hyperbola (the endpoints of the segments are $(0,\pm
\frac{1}{\sqrt{3}})$ and $(1,\pm \frac{1}{\sqrt{3}})$) and the
critical line. It has reduced the range of solutions of
$\zeta(s)=0$.

One may be tempted to reproduce the proof of Section~\ref{subsec4_2}
using this parametrization. Indeed, the critical line is already part of the set where
the solutions of $\zeta(s)$ must reside. The other set consists of
the two segments of the hyperbola. But since the first zeros of
$\zeta(s)$ are at an algebraic height of $\pm 14.13$, the two
segments of the hyperbola cannot contain any zero of $\zeta(s)$.
In other words, such a parametrization  would imply the Riemann
hypothesis. However, the attempted proof is untenable because we started with
a hypothetical artificial parametrization. There is no reason to
parametrize $\zeta(s)$ in the form (\ref{sec5-eq9}): it is not the result of well-defined series of $\zeta(s)$. And, any conclusion is true if the starting premise is false.

\end{example}

\begin{example}\label{sec5-example2}
In this example $x=1$ is a regular point and assumption (A2) is violated. We have a parametrization where the range of search of
the solution of $\zeta(s)=0$ is not reduced from the critical strip. 

Let $Z(s,x)=(s-1)\zeta(s,x)$ be our parametrization, where
$\zeta(s,x)$ is now the well-known Hurwitz zeta function. We use the power series expansion of 
$\zeta(s,x)$ around the point
$x=0$ given in \cite{klusch:lerch}:

\begin{equation}\label{sec5-eq12}
\zeta(s,a-t)=\sum_{n=0}^{\infty}\frac{\Gamma(s+n)}{\Gamma(s)n!}\zeta(s+n,a)t^n;
|t|<a.
\end{equation}

Choosing $a=1$ and writing $t=1-x$ leads to the desingularized
expansion

\begin{equation}\label{sec5-eq13}
(s-1)\zeta(s,x)=(s-1)\zeta(s)+s(s-1)\zeta(s+1)(1-x)[1+\cdots].
\end{equation}

For a nontrivial zero $s_0$ we must have

\begin{equation}\label{sec5-eq14}
\left |\frac{\zeta(s_0+1)}{\zeta(2-s_0)}\right|=\left
|\frac{s_0-1}{-s_0}\mathscr{R}(s_0)\right|.
\end{equation}

The exponent of $x-1$ in this expansion is 1 and equation
(\ref{sec5-eq7}) is trivially satisfied. Any non-trivial zero must
be a solution to equation (\ref{sec5-eq8}). It is  easy to verify
that any complex number $s_0$ with real part equal to
$\frac{1}{2}$ is a solution to (\ref{sec5-eq8}). But are there
other solutions to the same equation? We don't know. Hence, in some sense, 
expansion (\ref{sec5-eq13}) of the Hurwitz zeta function
is not very interesting: it only provides the finiteness of the
ratio $\left |\tfrac{\zeta(s_0+1)}{\zeta(2-s_0)}\right|$ which is
useful in itself but does not give a qualitative characterization
of the real part of $s_0$. The same conclusion applies to the parametrization of Knopp's series in 
Table~\ref{sec5-table1}.
\end{example}

\begin{example}\label{sec5-example3}
As a final example, we choose $Z_1(s,x)$ from
\cite{lazhar:zeros} defined by $\frac{Z(s,x)}{x}$: it is the parametrization of this paper but with the powers of $x$ reduced by $1$. Using a method very different from the method of this paper, the author obtained the following expansion \cite{lazhar:zeros}:

\begin{align}\label{sec5-eq15}
Z_1(s,x)&=(s-1)\zeta(s)-\frac{1}{\Gamma(s+1)}(1-x)\bigg(\frac{\log(1-x)}{-x}\bigg)^{s}[1+\cdots].
\end{align}

For a nontrivial zero $s_0$, we must have

\begin{eqnarray}\label{sec5-eq16}
\left |\frac{\Gamma(2-s_0)}{\Gamma(s_0+1)}\right|&=&\left |\frac{s_0-1}{-s_0}\mathscr{R}(s_0)\right| \text{and}\\
\re{(s_0)}&=&\re{(1-s_0)}.
\end{eqnarray}

Like the parametrization
(\ref{sec4-eq5}), parametrization (\ref{sec5-eq15}) is optimal in a sense that it
helped us reduce the range of solutions from the critical strip to
the critical line.

\end{example}

\section{Conclusion}\label{sec6}

 \noindent

We showed the existence of a useful new parametrized uniformly globally convergent series of the Riemann zeta function. As an application of the new parametrized series, we derived necessary conditions for a non-trivial zero of $\zeta(s)$. These necessary conditions are of a qualitative rather than the usual quantitative type. The Riemann hypothesis itself is of a qualitative character. By adding an extra parameter to the zeta function, information on the location of the zeros can be obtained without having to calculate these zeros. 

Other parametrizations of the zeta function may also be useful in determining
other necessary conditions for the non-trivial zeros. It is however important  to start with a legitimate and admissible parametrization because artificially constructed  parametrizations can be useless and may lead to the wrong conclusions.  For instance, further investigations of the known zeta series from Table~\ref{sec5-table1} may lead to interesting results.

Finally, it would be very useful and inspiring to apply and to  extend the algorithm described in \cite{vepstas} in order to graphically analyze the behavior of the parametrized zeta function described in this paper.

\end{document}